\input ./style/arxiv-vmsta.cfg
\documentclass[numbers,compress,v1.0.1]{vmsta}

\usepackage{vtexurl}

\volume{3}
\issue{4}
\pubyear{2016}
\firstpage{303}
\lastpage{313}
\doi{10.15559/16-VMSTA69}

\def\R{{ \mathbb{R}}}
\newtheorem{theorem}{Theorem}[section]
\newtheorem{lemma}[theorem]{Lemma}

\theoremstyle{definition}
\newtheorem{definition}[theorem]{Definition}

\def\Sup{\displaystyle\sup}
\def\Lim{\displaystyle\lim}

\startlocaldefs
\newcommand{\rrVert}{\Vert}
\newcommand{\llVert}{\Vert}
\newcommand{\rrvert}{\vert}
\newcommand{\llvert}{\vert}
\allowdisplaybreaks
\endlocaldefs

\begin{document}
\begin{frontmatter}

\title{Approximation of solutions of SDEs driven by a~fractional Brownian motion, under pathwise uniqueness}

\author[a]{\inits{O.}\fnm{Oussama}\snm{El Barrimi}\corref{cor1}\fnref{t1}}\email{Oussama.elbarrimi@gmail.com}
\cortext[cor1]{Corresponding author.}
\fntext[t1]{This author is supported by the CNRST ``Centre National pour la Recherche Scientifique et Technique'', grant No. I 003/034, Rabat, Morocco.}
\author[a,b]{\inits{Y.}\fnm{Youssef}\snm{Ouknine}}\email{ouknine@uca.ac.ma}

\address[a]{Cadi Ayyad University, Faculty of Sciences Semlalia,\\
Av. My Abdellah, 2390, Marrakesh, Morocco}
\address[b]{Hassan II Academy of Sciences and Technology Rabat, Morocco}

\markboth{O. El Barrimi, Y. Ouknine}{Approximation of solutions of SDEs driven by fBm, under pathwise uniqueness}

\begin{abstract}
Our aim in this paper is to establish some strong stability properties
of a~solution of a stochastic differential equation driven by a
fractional Brownian motion for which the pathwise uniqueness holds. The
results are obtained using Skorokhod's selection theorem.
\end{abstract}

\begin{keyword}
Fractional Brownian motion\sep
Stochastic differential equations
\MSC[2010] 60G15\sep60G22
\end{keyword}

\received{29 July 2016}
\revised{12 December 2016}
\accepted{13 December 2016}
\publishedonline{20 December 2016}
\end{frontmatter}

\section{Introduction}

Consider a fractional Brownian motion (fBm), a self-similar Gaussian
process with stationary increments. It was introduced by Kolmogorov
\cite{kol} and studied by Mandelbrot and Van Ness \cite{MN}. The fBm
with Hurst parameter $H\in(0,1)$ is a centered Gaussian process with
covariance function
\[
R_H(t,s)=E\bigl(B_{t}^{H}B_{s}^{H}
\bigr)=\frac{1}{2} \bigl(t^{2H}+s^{2H}-|t-s|^{2H}
\bigr).
\]
If $H=1/2$, then the process $B^{1/2}$ is a standard Brownian motion.
When \hbox{$H\neq\frac{1}{2}$}, $B^H$ is neither a semimartingale nor
a Markov process, so that many of the techniques employed in stochastic
analysis are not available for an fBm. The self-similarity and
stationarity of increments make the fBm an appropriate model for many
applications in diverse fields from biology to finance. We refer to
\cite{Nua10} for details on these notions.

Consider the following stochastic differential equation (SDE)
\begin{align}
\label{SDE1} %
\begin{cases}
dX_t = b(t,X_t)\,dt + dB^H_t,\\
X_0=x \in\R^d,
\end{cases} %
\end{align}
where $
b:[0,T]\times\R^d \rightarrow\R^d $ is a measurable function, and
$B^H$ is a $d$-dimensional fBm with Hurst parameter $H < 1/2$ whose
components are one-dimensional independent fBms defined on a
probability space $(\varOmega, \mathcal{F}, \{\mathcal{F}_t\}_{t\in[0,T]}, P)$,
where the filtration $\{\mathcal{F}_t\}_{t\in[0,T]}$ is generated by
$B^H_t$, $t\in[0,T]$, augmented by the $P$-null sets. It has been
proved in \cite{BNP} that if $b$ satisfies the assumption
\begin{equation}
\label{ass} b\in L_{\infty}^{1,\infty}:= L^\infty
\bigl([0,T]; L^1 \bigl(\mathbb{R}^{d}\bigr)\cap
L^\infty\bigl(\R^d\bigr)\bigr),
\end{equation}
for $H < \frac{1}{2(3d-1)}$, then Eq.~(\ref{SDE1}) has a unique strong
solution, which will be assumed throughout this paper.

Notice that if the drift coefficient is Lipschitz continuous, then
Eq.~(\ref{SDE1}) has a unique strong solution, which is continuous with
respect to the initial condition. Moreover, the solution can be
constructed using various numerical schemes.

Our purpose in this paper is to establish some stability results under
the pathwise uniqueness of solutions and under weak regularity
conditions on the drift coefficient~$b$. We mention that a considerable
result in this direction has been established in \cite{BMO} when an fBm
is replaced by a standard Brownian motion.

The paper is organized as follows. In Section~2, we introduce some
properties, notation, definitions, and preliminary results. Section~3
is devoted to the study of the variation of solution with respect to
the initial data. In the last section, we drop the continuity
assumption on the drift and try to obtain the same result as in Section~3.

\section{Preliminaries}
In this section, we give some properties of an fBm, definitions, and
some tools used in the proofs.

For any $H<1/2$, let us define the square-integrable kernel
\begin{equation*}
K_{H}(t,s)=c_{H} \Biggl[ \biggl(\frac{t}{s}
\biggr)^{H-\frac{1}{2}}- \biggl(H-\frac{1}{2} \biggr)s^{\frac{1}{2}-H}\int
_{s}^{t}(u-s)^{H-\frac
{1}{2}}u^{H-\frac{3}{2}}\,du
\Biggr], \quad t>s,
\end{equation*}
where $c_{H}= [\frac{2H}{(1-2H)\beta(1-2H,H+\frac{1}{2}))}
]^{1/2}$, $t>s$.

Note that
\begin{equation*}
\frac{\partial K_{H}}{\partial t}(t,s)=c_{H} \biggl(H-\frac{1}{2} \biggr)
\biggl(\frac{t}{s} \biggr)^{H-\frac{1}{2}}(t-s)^{H-\frac{3}{2}}.
\end{equation*}
Let $B^{H}=\{B_{t}^{H}, \ t\in[0,T]\}$ be an fBm defined on $(\varOmega,
\mathcal{F}, \{\mathcal{F}_t\}_{t\in[0,T]}, P)$. We denote by $\zeta$
the set of step functions on $[0,T]$. Let $\mathcal{H}$ be the Hilbert
space defined as
the closure of $\zeta$ with respect to the scalar product
\begin{equation*}
\langle\textbf{1}_{[0,t]}, \textbf{1}_{[0,s]}\rangle_{\mathcal{H}}=R_{H}(t,s).
\end{equation*}
The mapping $\textbf{1}_{[0,t]}\rightarrow B_{t}^{H}$ can be extended
to an isometry between $\mathcal{H}$ and the Gaussian subspace of
$L^2({\varOmega})$ associated with $B^{H}$, and such an isometry is
denoted by $\varphi\rightarrow B^{H}(\varphi)$.

Now we introduce the linear operator $K_{H}^{*}$ from $\zeta$ to
$L^{2}([0,T])$ defined by
\begin{equation*}
\bigl(K_{H}^{*}\varphi\bigr) (s)=K_{H}(b,s)
\varphi(s)+\int_{s}^{b}\bigl(\varphi (t)-\varphi(s)
\bigr)\frac{\partial K_{H}}{\partial t}(t,s)\,dt.
\end{equation*}
The operator $K_{H}^{*}$ is an isometry between $\zeta$ and
$L^{2}([0,T])$, which can be extended to the Hilbert space $\mathcal
{H}$.

Define the process $W= \{W_{t},t\in[0,T]\}$ by
\begin{equation*}
W_{t}=B^{H}\bigl(\bigl(K_{H}^{\ast}
\bigr)^{-1}\textbf{1}_{[0,t]}\bigr).
\end{equation*}
Then $W$ is a Brownian motion; moreover, $B^{H}$ has the integral representation
\begin{equation*}
B_t^{H}=\int_{0}^{t}K_{H}(t,s)\,dW(s).
\end{equation*}
We need also to define an isomorphism $K_H$ from $L^2([0,T])$ onto
$I_{0+}^{H+\frac{1}{2}}(L^2)$ associated with the kernel $K_H(t,s)$ in
terms of the fractional integrals as follows:
\[
(K_H \varphi) (s) = I_{0^+}^{2H} s^{\frac{1}{2}-H}
I_{0^+}^{\frac
{1}{2}-H}s^{H-\frac{1}{2}} \varphi, \quad\varphi\in
L^2\bigl([0,T]\bigr).
\]
Note that, for $\varphi\in L^2([0,T])$, $I_{0^+}^{\alpha}$ is the left
fractional Riemann-Liouville integral operator of order $\alpha$
defined by
\[
I_{0^+}^\alpha\varphi(x) = \frac{1}{\varGamma(\alpha)} \int
_0^x (x-y)^{\alpha-1}\varphi(y)\,dy,
\]
where $\varGamma$ is the gamma function (see \cite{DU} for details).

The inverse of $K_H$ is given by
\[
\bigl(K_H^{-1} \varphi\bigr) (s) = s^{\frac{1}{2}-H}
D_{0^+}^{\frac{1}{2}-H} s^{H-\frac{1}{2}} D_{0^+}^{2H}
\varphi(s), \quad\varphi\in I_{0+}^{H+\frac{1}{2}}\bigl(L^2
\bigr),
\]
where for $\varphi\in I_{0^+}^{H+\frac{1}{2}} (L^2)$, $D_{0^+}^{\alpha
}$ is the left-sided Riemann Liouville derivative of order $\alpha$
defined by
\[
D_{0^+}^{\alpha} \varphi(x)= \frac{1}{\varGamma(1-\alpha)} \frac{d
}{d x}
\int_0^x \frac{\varphi(y)}{(x-y)^{\alpha}}\,dy.
\]
If $\varphi$ is absolutely continuous (see \cite{NO}), then
\begin{align}
\label{inverseKH} \bigl(K_H^{-1} \varphi\bigr) (s) =
s^{H-\frac{1}{2}} I_{0^+}^{\frac{1}{2}-H} s^{\frac{1}{2}-H}
\varphi'(s).
\end{align}
\begin{definition}
On a given probability space $(\varOmega, \mathcal{F}, P)$, a process $X$
is called a strong solution to \eqref{SDE1} if
\begin{itemize}
\item[\rm(1)] $X$ is $\{\mathcal{F}_t\}_{t\in[0,T]}$ adapted, where $\{
\mathcal{F}_t\}_{t\in[0,T]}$ is the filtration generated by $B^H_t,
t\!\in[0,T]$;

\item[\rm(2)] $X$ satisfies \eqref{SDE1}.
\end{itemize}
\end{definition}
\begin{definition}
A sextuple $(\varOmega, \mathcal{F}, \{\mathcal{F}_t\}_{t\in[0,T]}, P, X,
B^H )$ is called a weak solution to \eqref{SDE1} if
\begin{itemize}
\item[\rm(1)] $(\varOmega, \mathcal{F}, P)$ is a probability space
equipped with the filtration $\{\mathcal{F}_t\}_{t\in[0,T]}$ that
satisfies the usual \xch{conditions;}{conditions,}

\item[\rm(2)] $X$ is an $\{\mathcal{F}_t\}_{t\in[0,T]}$-adapted
process, and $B^H$ is an $\{\mathcal{F}_t\}_{t\in[0,T]}$-fBm;

\item[\rm(3)] $X$ and $B^H$ satisfy \eqref{SDE1}.
\end{itemize}
\end{definition}
\begin{definition}[Pathwise uniqueness]
We say that pathwise uniqueness holds for Eq.~\eqref{SDE1}
if whenever $(X,B^{H})$ and $(\widetilde{X},{B^{H}})$
are two weak solutions of Eq.~\eqref{SDE1} defined on the
same probability space $ (\varOmega,\mathcal{F},(\mathcal{F}_{t})_{t\in
[0,T]},P )$,
then $X$ and $\widetilde{X}$ are indistinguishable.
\end{definition}

The main tool used in the proofs is Skorokhod's selection theorem given
by the following lemma.
\begin{lemma} \textup{(\cite{IW}, p.~9)}
Let $(S,\rho)$ be a complete separable metric space, and let $P$,
$P_n$, $n=1,2,\ldots$, be probability measures on $(S,\mathbb{B}(S))$
such that $P_n$ converges weakly to $P$ as $n \rightarrow\infty$.
Then, on a probability space $(\widetilde{\varOmega}, \widetilde{\mathcal
{F}},\widetilde{P})$, we can construct $S$-valued random variables $X$,
$X_n$, $n=1,2,\ldots$, such that:
\begin{itemize}
\item[\rm(i)] $P_n = \widetilde{P}^{X_n}$, $n=1,2,\ldots$, and $P =
\widetilde{P}^{X}$, where $\widetilde{P}^{X_n}$ and $\widetilde{P}^{X}$
are respectively the laws of ${X_n}$ and ${X}$;

\item[\rm(ii)] $X_n$ converges to $X$ $\widetilde{P}$-a.s.
\end{itemize}
\end{lemma}
We will also make use of the following result, which gives a criterion
for the tightness of sequences of laws associated with continuous processes.
\begin{lemma} \textup{(\cite{IW}, p.\ 18)} \label{tightness}
Let $\{X_{t}^{n}$, $t\in[0,T]\}$, $n=1,2,\ldots$, be a sequence of
$d$-dimensional continuous processes satisfying the following two conditions:
\begin{itemize}
\item[\rm(i)] There exist positive constants $M$ and $\gamma$ such that
$E[|X^n(0)|^{\gamma}]\leq M$ for every $n=1,2,\ldots$;
\item[\rm(ii)] there exist positive constants $\alpha$, $\beta$, $M_k$,
$k=1,2,\ldots$, such that, for every $n\geq1$ and all $t, s \in[0,k]$,
$k=1,2,\ldots$,
\[
E\bigl[\bigl|X_{t}^{n}-X_{s}^{n}\bigr|^{\alpha}
\bigr]\leq M_k |t-s|^{1+\beta}.
\]
\end{itemize}
Then, there exist a subsequence $(n_k)$, a probability space
$(\widetilde{\varOmega}, \widetilde{\mathcal{F}},\widetilde{P})$, and
$d$-dimen\-sional continuous processes $\widetilde{X}$, $\widetilde
{X}^{n_k}$, $k=1,2,\ldots$, defined on $\widetilde{\varOmega}$ such that
\begin{itemize}
\item[\rm(1)] The laws of $\widetilde{X}^{n_k}$ and $X^{n_k}$ coincide;
\item[\rm(2)] $\widetilde{X}^{n_k}_{t}$ converges to $\widetilde
{X}_{t}$ uniformly on every finite time interval $\widetilde{P}$-a.s.
\end{itemize}
\end{lemma}
\section{Variation of solutions with respect to initial conditions}
The purpose of this section is to ensure the continuous dependence of
the solution with respect to the initial condition when the drift $b$
is continuous and bounded. Note that, in the case of ordinary
differential equation, the continuity of the coefficient is sufficient
to ensure this dependence.

Next, we give a theorem that will be essential in establishing the
desired result.
\begin{theorem}\label{main}
Let $b$ be a continuous bounded function. Then, under the pathwise
uniqueness for SDE \eqref{SDE1}, we have
\[
\lim_{x\rightarrow x_0} E \Bigl[\sup_{0\leq t\leq T}\bigl\llvert
X_t (x)-X_t (x_0)\bigr\rrvert
^{2} \Bigr]=0.
\]
\end{theorem}
Before we proceed to the proof of Theorem~\ref{main}, we state the
following technical lemma.
\begin{lemma} \label{tight}
Let $X^n$ be the solution of \eqref{SDE1} corresponding to the initial
condition $x_n$. Then, for every $p>\frac{1}{2H}$, there exists a
positive constant $C_p$ such that, for all $s, t \in[0,T]$,
\[
E\bigl[\bigl|X^{n}_t-X^{n}_s\bigr|^{2p}
\bigr] \leq C_p |t-s|^{2pH}.
\]
\end{lemma}
\begin{proof}
Fix $s<t$ in $[0,T]$. We have
\begin{align*}
\bigl|X^{n}_t-X^{n}_s\bigr|^{2p} &\leq C_p \Biggl[\Biggl\llvert \int_s^t b\bigl(u,X^{n}_u\bigr)\,du\Biggr\rrvert ^{2p} + \bigl|B^H_t-B^H_s\bigr|^{2p}\Biggr].
\end{align*}
Due to the stationarity of the increments and the scaling property
of an fBm and the boundedness of $b$, we get that
\begin{align*}
E\bigl|X^{n}_t-X^{n}_s\bigr|^{2p} &\leq C_p \bigl[ |t-s|^{2p}+|t-s|^{2pH} \bigr]\\
&\leq C_p|t-s|^{2pH},
\end{align*}
which finishes the proof.
\end{proof}
Let us now turn to the proof of Theorem~\ref{main}.
\begin{proof}
Suppose that the result of the theorem is false. Then there exist a
constant $\delta>0$ and a sequence $x_n$ converging to $x_0$ such that
\[
\inf_n E \Bigl[\sup_{0\leq t\leq T}\bigl\llvert
X_t (x_n)-X_t (x_0)\bigr\rrvert
^{2} \Bigr]\geq\delta.
\]
Let $X^n$ (respectively, $X$) be the solution of (\ref{SDE1})
corresponding to the initial condition $x_n$ (respectively, $x_0$).
According to Lemma~\ref{tight}, the sequence $(X^{n},X,B^H)$ satisfies
conditions (i) and (ii) of Lemma~\ref{tightness}. Then, by Skorokhod's
selection theorem there exist a subsequence $\{n_k,k\geq1 \}$, a
probability space $(\widetilde{\varOmega}, \widetilde{\mathcal
{F}},\widetilde{P})$, and stochastic processes $(\widetilde{X},
\widetilde{Y}, \widetilde{B}^H)$, $(\widetilde{X^{k}}, \widetilde
{Y^{k}},\widetilde{B}^{H,k})$, $k\geq1 $, defined on $(\widetilde
{\varOmega}, \widetilde{\mathcal{F}},\widetilde{P})$ such that:
\begin{itemize}
\item [$(\alpha)$] for each $k\geq1$, the laws of $(\widetilde
{X^{k}}, \widetilde{Y^{k}},\widetilde{B}^{H,k})$ and $({X^{n_k}},
X,B^H)$ coincide;
\item[$(\beta)$] $(\widetilde{X}^{k}, \widetilde{Y}^{k}, \widetilde
{B}^{H,k})$ converges to
$(\widetilde{X}, \widetilde{Y}, \widetilde{B}^H)$ uniformly on every
finite time interval $\widetilde{P}$-a.s.
\end{itemize}
Thanks to property $(\alpha)$, we have, for $k\geq1$ and $t>0$,
\[
E\Biggl\llvert \widetilde{X}^{k}_{t} -x_k -
\int_0^t b\bigl(s,\widetilde
{X}^{k}_{s}\bigr)\,ds - \widetilde{B}^{H,k}_t
\Biggr\rrvert ^{2}=0.
\]
In other words, $\widetilde{X}^{k}_{t}$ satisfies the following SDE:
\[
\widetilde{X}^{k}_t= x_k + \int
_0^t b\bigl(s,\widetilde{X}^{k}_{s}
\bigr)\,ds + \widetilde{B}^{H,k}_t .
\]
Similarly,
\[
\widetilde{Y}^{k}_t= x_0 + \int
_0^t b\bigl(s,\widetilde{Y}^{k}_{s}
\bigr)\,ds + \widetilde{B}^{H,k}_t .
\]
Using $(\beta)$, we deduce that
\[
\lim_{k \rightarrow\infty}\int_0^t b\bigl(s,
\widetilde{X}^{k}_{s}\bigr)\,ds=\int_0^t
b(s,\widetilde{X}_{s})\,ds
\]
and
\[
\lim_{k \rightarrow\infty}\int_0^t b\bigl(s,
\widetilde{Y}^{k}_{s}\bigr)\,ds=\int_0^t
b(s,\widetilde{Y}_{s})\,ds
\]
in probability and uniformly in $t \in[0, T]$.

Thus, the processes $\widetilde{X}$ and $\widetilde{Y}$ satisfy the same
SDE on $(\widetilde{\varOmega}, \widetilde{\mathcal{F}},\widetilde{P})$
with the same driving noise $\widetilde{B}^{H}_t$ and the initial
condition $x_0$. Then, by pathwise uniqueness, we conclude that
$\widetilde{X}_{t}=\widetilde{Y}_{t}$ for all $t \in[0, T]$,
$\widetilde{P}$-a.s.

On the other hand, by uniform integrability we have that
\begin{eqnarray*}
\delta&\leq& \liminf_n E \Bigl[\max_{0\leq t\leq T}
\bigl\llvert X_t (x_n)-X_t (x_0)
\bigr\rrvert ^{2} \Bigr]
\\
&=&\liminf_k \widetilde{E} \Bigl[\max_{0\leq t\leq T}
\bigl\llvert \widetilde {X}^{k}_t-\widetilde{Y}^{k}_t
\bigr\rrvert ^{2} \Bigr]
\\
&\leq& \widetilde{E} \Bigl[\max_{0\leq t\leq T}\llvert \widetilde
{X}_t-\widetilde{Y}_t\rrvert ^{2} \Bigr],
\end{eqnarray*}
which is a contradiction. Then the desired result follows.
\end{proof}
\section{The case of discontinuous drift coefficient}
In this section, we drop the continuity assumption on the drift
coefficient and only assume that $b$ is bounded. The goal of this
section is to generate the same result as in Theorem~\ref{main} without
the continuity assumption.

Next, in order to use the fractional Girsanov theorem given in $\mbox
{\cite[Thm.~2]{NO},}$ we should first check that the conditions imposed
in the latter are
satisfied in our context. This will be done in the following lemma.
\begin{lemma}\label{gir}
Suppose that $X$ is a solution of SDE \eqref{SDE1}, and let $b$ be a
bounded function. Then the process
$v= K_H^{-1} ( \int_0^{\cdot} b(r, X_r) \,dr )$ enjoys the
following properties:
\begin{itemize}
\item[$(1)$] $v_s \in L^2([0,T]), \ P\text{-a.s.}$;
\item[$(2)$] $E [ \exp \{\frac{1}{2} \int_0^T |v_s|^2 \,ds \}  ] < \infty$.
\end{itemize}
\end{lemma}
\begin{proof}
(1) In light of \eqref{inverseKH}, we can write
\begin{align*}
|v_s| &= \bigl|s^{H-\frac{1}{2}} I_{0^+}^{\frac{1}{2}-H} s^{\frac{1}{2}-H} \bigl|b(s,X_s)\bigr| \bigr|\\
&= \frac{1}{\varGamma (\frac{1}{2}-H )} s^{H- \frac{1}{2}} \int_0^s(s-r)^{-\frac{1}{2}-H} r^{\frac{1}{2}-H} \bigl|b(r,X_r)\bigr|\,dr\\
&\leq \, \|b\|_\infty\frac{1}{\varGamma (\frac{1}{2}-H )} s^{H- \frac{1}{2}} \int_0^s (s-r)^{-\frac{1}{2}-H} r^{\frac{1}{2}-H}\,dr\\
&= \, \|b\|_\infty\frac{\varGamma (\frac{3}{2}-H )}{\varGamma(2-2H  )}s^{\frac{1}{2}-H}\\
&\leq\, \|b\|_\infty\frac{\varGamma (\frac{3}{2}-H )}{\varGamma(2-2H  )}T^{\frac{1}{2}-H},
\end{align*}
where $\|\cdot\|_{\infty}$ denotes the norm in $L^{\infty}([0,T];
L^\infty(\R^d))$.

As a result, we get that
\begin{align*}
\int_{0}^{T} |v_{s}|^2 \,ds
<\infty, \quad P\text{-a.s.}
\end{align*}
(2) The second item is obtained easily by the following estimate:
\begin{align*}
E& \Biggl[ \exp \Biggl\{\frac{1}{2} \int_0^T
\llvert v_s\rrvert ^2 \,ds \Biggr\} \Biggr] \leq\exp \biggl
\{ \frac{1}{2} C_H T^{2(1-H)} \|b\|_\infty
^2 \biggr\},
\end{align*}
where $C_H=\frac{\varGamma (\frac{3}{2}-H )^2}{\varGamma (2-2H
 )^2}$,
which finishes the proof.
\end{proof}
Next, we will establish the following Krylov-type inequality that will
play an essential role in the sequel.
\begin{lemma}
\label{krylov1}
Suppose that $X$ is a solution of SDE \eqref{SDE1}. Then, there exists
$\beta>1+dH$ such that, for
any measurable nonnegative function $g:[0,T]\times\mathbb R^d \mapsto
\mathbb R^d
_+$, we have
\begin{eqnarray}
\label{krylovI} E\int_{0}^{T}g(t,X_{t})
\,dt \leq M \Biggl(\int_{0}^{T}\int
_{\R^d} g^{\beta}(t,x)\,dx\,dt \Biggr)^{1/\beta},
\end{eqnarray}
where $M$ is a constant depending only on $T$, $d$, $\beta$, and $H$.
\end{lemma}
\begin{proof}
Let $W$ be a $d$-dimensional Brownian motion such that
\begin{equation*}
B^{H}_t=\int_0^t
K_H(t,s)\,dW_s.
\end{equation*}
For the process $v$ introduced in Lemma~\ref{gir}, let us define
$\widehat{P}$ by
\begin{eqnarray*}
\label{ZT} \frac{d\widehat{P} }{dP}=\exp \Biggl\{ -\int_{0}^{T}v_{t}
\,dW_{t} - \frac{1}{2}\int_{0}^{T}v_{t}^{2}
\,dt \Biggr\}:=Z_T^{-1}.
\end{eqnarray*}
Then, in light of Lemma~\ref{gir} together with the fractional Girsanov
theorem $\mbox{\cite[Thm.~2]{NO}}$, we can conclude that $\widehat{P}$
is a probability measure under which the process $X-x$ is an fBm.

Now, applying H\"{o}lder's inequality, we have
\begin{align}
E\int_{0}^{T}g(t,X_{t})\,dt&= \widehat E \Biggl\{Z_T\int_0^Tg(t,X_{t})\,dt\Biggr\}\nonumber\\[-3pt]
&\leq C \bigl\{\widehat{E} \bigl[Z_T^\alpha \bigr] \bigr\}^{1/\alpha} \Biggl\{\widehat{E}\int_{0}^{T}g^{\rho}(t,X_{t})\,dt \Biggr\}^{1/\rho},\label{ineq}
\end{align}
where $1/\alpha+1/\rho=1$, and $C$ is a positive constant depending
only on $T$, $\alpha$, and $\rho$.

From $\mbox{\cite[Lemma 4.3]{BNP}}$ we can see that $\widehat
E[Z_T^\alpha]$ satisfies the following property:
\begin{equation}
\label{ineq-2} \widehat E\bigl[Z_T^\alpha\bigr] \leq
C_{H,d,T}\bigl(\|b\|_\infty\bigr) < \infty,
\end{equation}
where $C_{H,d,T}$ is a continuous increasing function depending only on
$H$, $d$, and~$T$.

On the other hand, applying again H\"{o}lder's inequality with $1/\gamma
+1/\gamma^{\prime}=1$
and $\gamma>dH+1$, we obtain
\begin{align}
\label{previous-inequality}
\widehat{E}\int_{0}^{T}g^{\rho}(t,X_{t})\,dt&=\int_{0}^{T} \int_{\R^d}g^{\rho}(t,y)\bigl(2\pi t^{2H}\bigr)^{-d/2}\exp^{-\|y-x\|^{2}/2t^{2H}}\,dy\,dt\nonumber\\[-3pt]
&\leq \Biggl(\int_{0}^{T}\int_{\R^d} \bigl(2\pi t^{2H}\bigr)^{-d\gamma^{\prime}/2}\exp^{-\gamma^{\prime}\|y-x\|^{2}/2 t^{2H}}\,dy\,dt \Biggr)^{1/\gamma^{\prime}}\nonumber\\[-3pt]
&\quad \times \Biggl(\int_{0}^{T}\int_{\R^d} g^{\rho\gamma}(t,y)\,dy\,dt \Biggr)^{1/\gamma}.
\end{align}
A direct calculation gives
\begin{eqnarray*}
\int_{\R^d}\bigl(2\pi t^{2H}\bigr)^{-d\gamma^{\prime} /2}
\exp^{-\gamma^{\prime}\|
y-x\|^{2}/2t^{2H}}\,dy =(2\pi)^{d/2-d\gamma^{\prime} /2} \bigl(\gamma^{\prime}
\bigr)^{-d/2}t^{(1-\gamma^{\prime})\,dH}.
\end{eqnarray*}
Plugging this into \eqref{previous-inequality}, we get
\begin{align*}
\widehat{E}\int_{0}^{T}g^{\rho}(t,X_{t})\,dt&\leq \Biggl(\int_{0}^{T}(2\pi)^{d/2-d\gamma^{\prime} /2}\bigl(\gamma ^{\prime} \bigr)^{-d/2}t^{(1-\gamma^{\prime})\,dH}\,dt\Biggr)^{1/\gamma^{\prime}}\\[-3pt]
&\quad \times \Biggl(\int_{0}^{T}\int_{\R^d} g^{\rho\gamma}(t,y)\,dy\,dt \Biggr)^{1/\gamma}\\[-3pt]
&\leq \bigl((2\pi)^{d/2-d\gamma^{\prime} /2} \bigl(\gamma^{\prime}\bigr)^{-d/2} \bigr)^{1/\gamma^{\prime}} \Biggl(\int_{0}^{T}t^{(1-\gamma^{\prime})\,dH}\,dt\Biggr)^{1/\gamma^{\prime}}\\
&\quad \times \Biggl(\int_{0}^{T}\int_{\R^d} g^{\rho\gamma}(t,y)\,dy\,dt \Biggr)^{1/\gamma}\\
&\leq C\bigl(\gamma^{\prime},T,d,H\bigr) \Biggl(\int_{0}^{T}\int_{\R^d} g^{\rho\gamma}(t,y)\,dy\,dt \Biggr)^{1/\gamma}.
\end{align*}
Finally, combining this with \eqref{ineq} and \eqref{ineq-2}, we get
estimate \eqref{krylovI} with $\beta=\rho\gamma$. The proof is now complete.
\end{proof}
Now we are able to state the main result of this section.
\begin{theorem}
If the pathwise uniqueness holds for Eq.~\eqref{SDE1}, then without the
continuity assumption on the drift coefficient, the conclusion of
Theorem~\ref{main} remains valid.
\end{theorem}
\begin{proof}
The proof is similar to that of Theorem~\ref{main}. The only difficulty
is to show that
\[
\lim_{k \rightarrow\infty}\int_0^t b\bigl(s,
\widetilde{X}^{k}_{s}\bigr)\,ds=\int_0^t
b(s,\widetilde{X}_{s})\,ds
\]
in probability. In other words,
for $\epsilon>0$, we will show that
\begin{eqnarray}
\label{desired-result} \limsup_{k \rightarrow\infty} P \Biggl[\Biggl\llvert \int
_0^t \bigl(b\bigl(s,\widetilde{X}^{k}_{s}
\bigr)- b(s,\widetilde{X}_{s}) \bigr) \,ds \Biggr\rrvert > \epsilon
\Biggr]=0.
\end{eqnarray}
Let us first define
\[
b^{\delta} (t,x)= \delta^{-d} \phi(x/\delta) \ast b(t,x) ,
\]
where $\ast$ denotes the convolution on $\R^d$, and $\phi$ is an
infinitely differentiable function with support in the unit ball such
that $\int\phi(x)\,dx = 1$.

Applying Chebyshev's inequality, we obtain
\begin{align*}
&P \Biggl[\Biggl\llvert \int_0^t\bigl(b\bigl(s,\widetilde{X}^{k}_{s}\bigr)- b(s,\widetilde{X}_{s}) \bigr) \,ds\Biggr\rrvert > \epsilon \Biggr]\nonumber\\
&\quad \leq \frac{1}{\epsilon^2} E \Biggl[\int_0^t\bigl| b\bigl(s,\widetilde{X}^{k}_{s}\bigr)- b(s,\widetilde{X}_{s})\bigr|^{2} \,ds \Biggr]\nonumber\\
&\quad \leq \frac{4}{\epsilon^2} \Biggl\{ E \Biggl[\int_0^t\bigl|b\bigl(s,\widetilde {X}^{k}_{s}\bigr)- b^{\delta}\bigl(s,\widetilde{X}^{k}_{s}\bigr)\bigr|^{2} \,ds\Biggr]\nonumber\\
&\qquad + E \Biggl[\int_0^t\bigl|b^{\delta} \bigl(s,\widetilde{X}^{k}_{s}\bigr)-b^{\delta} (s,\widetilde{X}_{s})\bigr|^{2} \,ds \Biggr]\nonumber\\
&\qquad + E \Biggl[\int_0^t\bigl|b^{\delta} (s,\widetilde{X}_{s})- b(s,\widetilde{X}_{s})\bigr|^{2}\,ds \Biggr] \Biggr\}\nonumber\\
&\quad = \frac{4}{\epsilon^2}(J_1+J_2+J_3).
\end{align*}
From the continuity of $b^{\delta}$ in $x$ and from the convergence of
$\widetilde{X}^{k}_{s}$ to $\widetilde{X}_{s}$ uniformly on every
finite time interval $\widetilde{P} $ a.s.\ it follows that $J_2$
converges to 0 as $k \rightarrow\infty$ for every $\delta>0$.\vadjust{\eject}

On the other hand, let $\theta: \mathbb R^d\rightarrow\mathbb R_+ $
be a smooth truncation function
such that $\theta(z)=1$ in the unit ball and $\theta(z)=0$ for
$|z|>1$.

By applying Lemma~\ref{krylov1} we obtain
\begin{align}
\label{in2} J_1 &= E\int_{0}^{t}\theta\bigl(\widetilde{X}^{k}_{s}/R\bigr) \bigl| b^{\delta}\bigl(s,\widetilde {X}^{k}_{s}\bigr) - b\bigl(s,\widetilde{X}^{k}_{s}\bigr) \bigr|^2 \,ds\nonumber\\
&\quad + E\int_{0}^{t} \bigl(1 - \theta\bigl(\widetilde{X}^{k}_{s}/R\bigr) \bigr) \bigl|b^{\delta}\bigl(s,\widetilde{X}^{k}_{s}\bigr)-b\bigl(s,\widetilde{X}^{k}_{s}\bigr) \bigr|^2 \,ds\nonumber\\
&\leq N \bigl\llVert b^{\delta}- b \bigr\rrVert _{\beta,R}+2CE\int_{0}^{t} \bigl(1-\theta\bigl(\widetilde{X}^{k}_{s}/R\bigr) \bigr)\,ds,
\end{align}
where $N$ does not depend on $\delta$ and $k$, and $\|\cdot\|_{\beta
,R}$ denotes the norm in
$L^{\beta}([0,T]\times B(0,R))$.

The last expression in the right-hand side of the last inequality
satisfies the following estimate:
\begin{equation}
\label{in2'} E\int_{0}^{t} \bigl(1 - \theta\bigl(\widetilde{X}^{k}_{s}/R\bigr) \bigr)\,ds \leq
\Sup_{k\geq1} P \Bigl[\Sup_{s\leq t}\bigl|\widetilde{X}^{k}_{s}\bigr|>R\Bigr].
\end{equation}
But we know that $\sup_{k\geq1} E [\sup_{s\leq t}|\widetilde
{X}^{k}_{s}|^{p} ] < \infty$ for all $p>1$, and thus
\begin{equation}
\label{sup} \Lim_{R \rightarrow\infty}\Sup_{k\geq1} P
\Bigl[\Sup_{s\leq t}\bigl|\widetilde{X}^{k}_{s}\bigr|>R \Bigr] = 0.
\end{equation}
Substituting estimate \eqref{in2'} into \eqref{in2}, letting $\delta
\rightarrow0$, and using \eqref{sup}, we deduce that the convergence
of the term $J_1$ follows.

Finally, since estimate \eqref{in2'} also holds for $\widetilde X$, it
suffices to use the same arguments as before to obtain the convergence
of the term $J_3$, which completes the proof.
\end{proof}
\section*{Acknowledgements} We thank the reviewer for his thorough
review and highly appreciate the comments and
suggestions, which significantly contributed to improving the quality
of the paper.

%

\end{document}